\documentclass{amsproc}

\newtheorem{theorem}{Theorem}[section]
\newtheorem{lemma}[theorem]{Lemma}
\newtheorem{prop}[theorem]{Proposition}

\theoremstyle{definition}
\newtheorem{definition}[theorem]{Definition}

\theoremstyle{remark}
\newtheorem*{remark}{Remark}

\numberwithin{equation}{section}

\theoremstyle{remark}

\newtheorem*{spf}{Sketch of proof}

\theoremstyle{plain}

\numberwithin{equation}{section}

\begin{document}
\title{On the Center of Mass in General Relativity}
\author{Lan--Hsuan Huang}
\address{Department of Mathematics, Columbia University, New York, NY 10027.}
\email{lhhuang@math.columbia.edu}
\begin{abstract}
The classical notion of center of mass for an isolated system in general relativity is derived from the Hamiltonian formulation and represented by a flux integral at infinity. In contrast to mass and linear momentum which are well-defined for asymptotically flat manifolds, center of mass and angular momentum seem less well-understood, mainly because they appear as the lower order terms  in the expansion of the data than those which determine mass and linear momentum. This article summarizes some of the recent developments concerning center of mass and its geometric interpretation using the constant mean curvature foliation near infinity. Several equivalent notions of center of mass are also discussed. 
\end{abstract}
\thanks{The author was partially supported by the NSF through DMS-1005560.}


\maketitle

\section{Introduction}
General relativity is the theory of the Einstein equation. Let $(\mathcal{M}, \tilde{g})$ be a four dimensional manifold with a Lorentz metric $\tilde{g}$. The Einstein equation says
\[	
	G := Ric(\tilde{g}) - \frac{1}{2} R(\tilde{g})  \tilde{g}  = T,
\] 
where $Ric(\tilde{g})$ and  $R(\tilde{g})$ are the Ricci tensor and the scalar curvature of $\tilde{g}$ respectively. The Einstein tensor $G$ corresponds to the gravitational field of spacetime, while the stress-energy-momentum tensor $T$ is related to the matter content of spacetime.

If $M$ is a spacelike hypersurface in $(\mathcal{M}, \tilde{g})$, we denote the induced Riemannian metric by $g$ and the induced second fundamental form by $K$. It follows from the Gauss and Codazzi equations combined with the Einstein equation that $g$ and $K$ must satisfy the following constraint equations:
\begin{align*}
     	 	   R_g - | K |_g^2 + ( \textup{tr}_g K )^2 & =2 \mu,\\
			\textup{div}_g (K - \textup{tr}_g K g) & = J,
\end{align*}
where $R_g$ is the scalar curvature of $g$, and $\mu$ and $J$ are respectively the local energy density  and the local momentum density of matter defined by $T$. It is more convenient to introduce the momentum tensor $\pi := K - (\textup{tr}_g K) g$. Then the constraint equations become:
\begin{align*}
	R_g - | \pi |_g^2 + \frac{1}{2} ( \textup{tr}_g \pi )^2 &= 2 \mu,\\
	\textup{div}_g \pi &=  J.
\end{align*}
 The triple $(M, g, \pi)$ satisfying the constraint equations is called an \emph{initial data set}. An important class of initial data sets which model isolated systems is the class of asymptotically flat manifolds. An initial data set $(M, g,\pi)$ is called \emph{asymptotically flat} at the decay rate $q$ greater than $1/2$, if there exists a coordinate system $\{x\}$ outside some compact set, say $B_{R_0}$, such that $g_{ij} (x) = \delta_{ij} + O(|x|^{-q})$ and $\pi_{ij} (x) = O(|x|^{-1-q})$ with the corresponding decay rates on higher derivatives. If $\mu$ and $J$ vanish, $(M,g,\pi)$ is called \emph{vacuum}. Here we allow nonzero $\mu$ and $J$ if they satisfy the decay condition $\mu = O(|x|^{-2-2q})$ and $J = O(|x|^{-2-2q})$.

 The first major investigation into general relativity as a dynamical system was by Arnowitt, Deser, and Misner \cite{ADM61}. Later, various people further studied the correct formulations of total energy (the ADM mass), linear momentum, center of mass, and angular momentum for asymptotically flat manifolds (e.g. \cite{ DeWitt67, RT74, BO87}) and their well-definedness (e.g. \cite{Bartnik87, Chrusciel88, CW08, Huang09a}). 

The mass and linear momentum are defined by the flux integrals at infinity:
\begin{align} 
	\label{eqn:mass}	m &= \frac{1}{16 \pi } \lim_{r \rightarrow \infty} \int_{ |x| = r } \sum_{ i,j} \left( g_{ij,i} - g_{ii,j}\right) \frac{x^j}{ r } \, d\sigma_0,\\
	\notag	P^j &= \frac{1}{8\pi} \lim_{r \rightarrow \infty}  \int_{ |x| = r } \sum_{ i} \pi_{ij} \frac{x^i}{r} \, d\sigma_0.
\end{align}
The mass and linear momentum are well-defined because they are independent of the chosen asymptotically flat coordinates as proven by Bartnik \cite{Bartnik87} and Chru\'{s}ciel \cite{Chrusciel86}. The positive mass theorem by Schoen--Yau \cite{SY79,SY81b} and by Witten \cite{Witten81} states: if $(M,g,\pi)$ is asymptotically flat with the dominant energy condition $\mu \ge | J |_g$, then $m$ is non-negative. Moreover, $m = 0$ if and only if $(M,g,\pi)$ is isometric to a hypersurface in Minkowski spacetime.

Regge and Teitelboim \cite{RT74} observed that center of mass and angular momentum are not generally defined for asymptotically flat manifolds unless some parity condition at infinity is imposed (see also  \cite{AH, Chrusciel87} for other proposed conditions). Given an asymptotically flat coordinate system, we denote by $f^{\textup{odd} } = (f(x) - f(-x))/2$ and $f^{\textup{even} } = (f(x) + f(-x))/2$ the odd and even parts of a function $f$. They are defined outside $B_{R_0}$ when the coordinates $\{x\}$ exist.
\begin{definition}
$(M,g,\pi)$ is asymptotically flat satisfying the RT condition if $(M,g,\pi)$ is asymptotically flat at a decay rate $q>1/2$ and $g, \pi$ are respectively asymptotically even/odd, i.e. 
\begin{align*}
	& g_{ij}^{\textup{odd}}= O(|x|^{- 1 - q}) &\pi_{ij}^{\textup{even}}= O(|x|^{-2-q}),\\
	& \mu^{\textup{odd}} = O(|x|^{-3-2q}) &J^{\textup{odd}} = O(|x|^{-3-2q}), 
\end{align*}
and similarly on higher derivatives. When we say that $(M,g)$ is asymptotically flat satisfying the RT condition, we mean $(M,g)$ satisfying the corresponding conditions by letting $\pi=0$. 
\end{definition}
While all known physical examples of asymptotically flat manifolds satisfy the RT condition, there exist mathematical solutions to the vacuum constraint equations which violate the RT condition, and their center of mass and angular momentum are ill-defined \cite{BO87, Huang10b}. However, these examples may not be physical in the sense that the evolution of the initial data is unknown. Nevertheless, under the RT condition, the Hamiltonian formulation of center of mass and angular momentum proposed by Regge--Teitelboim \cite{RT74} and by Beig--{\'O} Murchadha \cite{BO87} are defined by
\begin{align} \label{def:ctm}
		C^l &= \frac{1}{16\pi m } \lim_{r \rightarrow \infty} \left[ \int_{ |x| = r } x^l\sum_{i,j} \left( g_{ij,i} - g_{ii, j} \right) \frac{x^j}{ r} \, d\sigma_0\right. \\
		&\qquad \qquad \qquad \qquad \left.- \int_{|x| = r} \sum_i \left( g_{i l} \frac{x^i}{ r } - g_{ii} \frac{x^l } {r} \right) \, d\sigma_0 \right], \notag\\
		J^p &=  \frac{1}{8\pi m } \lim_{r \rightarrow \infty}  \int_{ |x| = r } \sum_{i,j} \pi_{ij} Y_{(p)}^i \frac{ x^j}{r } \, d\sigma_0, \notag
\end{align}
where $Y_{(p)} = \frac{\partial}{\partial x^p} \times x$ (cross product) are the rotation vector fields.

In addition to the Hamiltonian formulation of center of mass \eqref{def:ctm}, there are other proposed definitions such as \eqref{CTM} and \eqref{CENTER} below. A geometric interpretation of center of mass uses the constant mean curvature foliation in $M$ near infinity proposed  by Huisken and Yau \cite{HY96}. They used the volume preserving mean curvature flow to construct the constant mean curvature foliation, under the condition that $(M,g)$ is \emph{strongly asymptotically flat}, i.e. $g_{ij} = (1 + \frac{2m}{|x|})\delta_{ij}+ O(|x|^{-2})$. Similar results were proven by Ye using the inverse function theorem \cite{Ye96}. Metzger generalized the results to small perturbations of the strongly asymptotically flat data \cite{Metzger07}. Those authors also proved that the foliation is unique under some conditions.  A more general uniqueness result was obtained by Qing and Tian \cite{QT07}. For strongly asymptotically flat metric which is conformally flat near infinity, Corvino and Wu proved the geometric center of Huisken--Yau's foliation is equal to center of mass \cite{CW08}, and we later removed the condition of being conformally flat \cite{Huang09a}.

However, there are various interesting physical solutions to the constraint equations which are not strongly asymptotically flat. For example, the family of Kerr solutions is a family of the exact solutions to the vacuum Einstein equation which model the rotating black holes with angular momentum. They are not strongly asymptotically flat, but they do satisfy the RT condition. 

In \cite{Huang10a}, we generalized the earlier results of the constant mean curvature foliation to asymptotically flat manifolds with the RT condition, which is a natural condition to impose when center of mass is discussed. Furthermore, the foliation that we constructed is unique under some conditions. From our construction, we can show that the foliation is asymptotically concentric, whose geometric center is consistent with the classical notion of center of mass. 

A new ingredient in constructing the foliation is that for any asymptotically flat manifold with the RT condition, one can find a canonical family of the {\it approximate spheres} $\mathcal{S}(p,R)$. They are constructed from perturbing the coordinate spheres $S_R(p):=\{x : |x-p| = R\}$ centered at $p$ of radius $R$. The approximate spheres are better adapted to the asymptotic symmetry than the coordinate spheres. Although their mean curvature is not exactly constant, they share many nice properties with the constant mean curvature surfaces constructed from them. For example, the leading order terms of the lowest two eigenvalues of the stability operator are $-2/R^2$ and $6m/R^3$. Moreover, $\{\mathcal{S}(p,R)\}$ form a foliation centered at a fixed $p$ for $R$ large when $m>0$. In many applications, the approximate spheres should be as good as the constant mean curvature surfaces, except that they are insensitive to center of mass. 

In order to construct the constant mean curvature surfaces from the approximate spheres, we have to choose the center $p$ of $\mathcal{S}(p,R)$, and a suitable choice of $p$ would allow us to find a nearby constant mean curvature surface. The dependence of the choice of $p$ on center of mass is based on the new observation of the following identity
\begin{align*} 
		\int_{x\in S_R(p)} (x^l - p^l)  H_S \, d\sigma_0 = 8 \pi m (p^l - C^l) + O(R^{1-2q}) \qquad \textup{for } l = 1, 2, 3,
\end{align*}
where $H_S$ is the mean curvature of $S_R(p)$. This identity can be thought as an alternative definition of center of mass. It also enables us to show that the geometric center of each constant mean curvature surface is close to center of mass, and the limit of the geometric centers is precisely center of mass.

The analytic underpinnings of the above identity rely crucially on the density theorem (Theorem \ref{DenThm2} below) that we obtained in \cite{Huang09a}. The density theorem is a refinement of the density theorem established by Corvino and Schoen \cite{CS06}. An initial data set $(M, g,\pi)$ is said to  have \emph{harmonic asymptotics} if, outside a compact set, $g = u^4 \delta$ and $\pi = u^2 \left( L_{X} \delta - (\textup{div} X)\delta \right)$ for some function $u$ and vector field $X$. Because of the constraint equations, $u$ and $X_i$ are harmonic up to the lower order terms.  Moreover,  the quantities $m, P, C, J$ can be read off of the expansion of $u$ and $X$. The density theorems involve finding suitable spaces so that initial data sets with harmonic asymptotics form a dense subset of  asymptotically flat data sets in a weighted Sobolev topology, and the physical quantities are continuous in that topology. 

The method to construct the approximate sequence of the initial data sets in the density theorems can be employed to generate more initial data sets, often with specified asymptotics and physical properties. Several further applications can be found in, for example \cite{Huang10b, HSW10}.

The article is organized as follows. We present the density theorems in Section 2, in which a slightly more general proposition is proven. In Section 3, we sketch the construction of the constant mean curvature foliation. Several equivalent notions of center of mass are then discussed in Section 4. 

\section{The density theorems}
In the previous section, the definition of asymptotically flat manifolds involves the pointwise regularity on $g$ and $\pi$. A slightly weaker regularity condition on asymptotically flat manifolds is defined by weighted Sobolev spaces (see more detailed discussions in, for example \cite{Bartnik87}). 

\begin{definition}[Weighted Sobolev spaces]
For a non-negative integer $k$, a non-negative real number $p$, and a real number $q$, we say $ f \in W^{k,p}_{-q} (M) $ if 
$$
		\| f\|_{W^{k,p}_{-q} (M)} := \left( \int_M \sum_{|\alpha| \le k} \left( \big| D^{\alpha} f \big| \rho ^{ | \alpha| + q }\right)^p \rho^{-3} \, d\textup{vol}  \right)^{\frac{1}{p}} < \infty,
$$
where $\alpha$ is a multi-index and $\rho$ is a continuous function with $\rho = |x|$ when the coordinates $\{x\}$ are defined. When $p = \infty$, 
$$
		\| f\|_{W^{k,\infty}_{-q} (M)} = \sum_{|\alpha | \le k} ess \sup_{M } | D^{\alpha} f | \rho^{ |\alpha| + q}.
$$
\end{definition}
Using the weighted Sobolev norms, we say that $(M,g,\pi)$ is asymptotically flat at the decay rate $q$ if $(g-\delta, \pi) \in W^{2,p}_{-q}(M) \times W^{1,p}_{-1-q} (M)$, and in addition $(M,g,\pi)$ satisfies the RT condition if $( g^{\textup{odd}} , \pi^{\textup{even}} ) \in W^{2,p}_{-1-q}(M \setminus B_{R_0}) \times W^{1,p}_{-2-q} (M \setminus B_{R_0})$.

The proof of the positive mass theorem by Schoen and Yau \cite{SY79} was originally stated for strongly asymptotically flat data.  Later, they extended the previous proof to allow general asymptotically flat data by a density argument \cite{SY81a}. They observed that, given a scalar-flat metric $g$, there exists a sequence of scalar-flat metrics with harmonic asymptotics whose mass converges to the mass of $g$. To generalize their density result, one has to consider not only the scalar curvature equation but the full constraint equations, which is more subtle. The following theorem by Corvino and Schoen is the analogue for the full constraint equations.

\begin{theorem}[Corvino--Schoen \cite{CS06}]\label{DenThm}
Suppose $p>3/2$ and $q\in (1/2, 1)$. Let $ ( g_{ij} - \delta_{ij} , \pi_{ij} ) \in W^{2,p}_{-q} ( M )  \times W^{1,p}_{-1 -q} ( M ) $ be a vacuum initial data set. There is a sequence of solutions $ ( \bar{ g }_k , \bar{ \pi }_k ) $ with harmonic asymptotics. Given any $ \epsilon > 0 $, there exist $k_0 > 0$ so that
$$
		\| g - \bar{g}_k \|_{W^{2,p}_{-q} ( M ) } \le \epsilon , \quad \| \pi - \bar{\pi}_k \|_{W^{1,p}_{-1-q} ( M ) } \le \epsilon, \quad \textup{for all } k \ge k_0.
$$
Moreover, the mass and linear momentum of $ ( \bar{ g }_k, \bar{ \pi }_k)$ are within $ \epsilon $ of those of $ ( g, \pi )$. 
\end{theorem}
The theorem says that the solutions with harmonic asymptotics are dense among general asymptotically flat solutions. Moreover, the mass and linear momentum, which can be explicitly expressed in the expansion of the solutions with harmonic asymptotics, converge to those of the original initial data set in these weighted Sobolev spaces. However, in the above theorem the center of mass and angular momentum may not converge, neither are they defined generally for asymptotically flat manifolds. In the following, we show that in some more refined weighted Sobolev spaces, solutions with harmonic asymptotics form a dense subset inside asymptotically flat solutions with the RT condition so that the center of mass and angular momentum are continuous in that topology. Also, we can remove the condition that $(g,\pi)$ is vacuum.

\begin{theorem}[\cite{Huang09a}]{\label{DenThm2}}
Suppose $p>3/2$ and $q\in (1/2, 1)$. Let $ ( g_{ij} - \delta_{ij} , \pi_{ij} ) \in W^{2,p }_{ - q} (M) \times W^{ 1,p }_{ -1- q} (M) $ be an initial data set. Suppose that $(g_{ij}^{\textup{odd}} , \pi_{ij}^{\textup{even} }) \in W^{2,p }_{ -1 - q} ( M \setminus B_{ R_0 })  \times W^{1,p }_{ -2- q} ( M \setminus B_{ R_0 })$.

There is a sequence of data $ ( \bar{ g }_k , \bar{ \pi }_k ) $ with harmonic asymptotics.  Given any $ \epsilon > 0 $, $ ( \bar{ g }_k , \bar{ \pi }_k ) $  is within an $\epsilon$-neighborhood of $ (g, \pi ) $ in $W^{2,p }_{ - q} (M) \times W^{ 1,p }_{ -1 - q} (M) $ for $k$ large as in the above theorem. Moreover, there exist $R$ and $k_0 = k_0(R)$ so that
$$
		 \| \bar{g}_k^{\textup{odd}} \|_{ W_{ -1 - q }^{2,p} ( M \setminus B_R ) } \le C, \quad \| \bar{ \pi }_k^{\textup{even}} \|_{  W_{ -2 - q }^{1,p} ( M \setminus B_R ) } \le C, \quad \textup{for all } k \ge k_0.
$$
Furthermore, mass, linear momentum, center of mass, angular momentum of $ ( \bar{ g }_k, \bar{ \pi }_k)$ are within $ \epsilon $ of those of $ ( g, \pi )$.
\end{theorem} 
The existence statement of the sequence of solutions with harmonic asymptotics in the above two theorems involves solving an elliptic system for $u_k$ and $X_k$ by the inverse function theorem. 
In the proof of Theorem \ref{DenThm2}, we study more carefully the elliptic system that $u_k, X_k$ and $u_k^{\textup{odd}}, X_k^{\textup{odd}}$ satisfy. Then we apply the boost trap argument to improve the decay rates of $u_k^{\textup{odd}}$ and $X_k^{\textup{odd}}$ and prove that $\{ (\bar{g}_k, \bar{\pi}_k )\}$ also satisfy the RT condition. The full proofs of the above two theorems are rather technical. Some arguments there can prove the following proposition which is slightly more general (in the sense that the sequence $\{(\bar{g}_k, \bar{\pi}_k) \}$ need not have harmonic asymptotics) and is of independent interest.
\begin{prop} \label{prop:cont}
Suppose that $p\ge 1$ and $q > 1/2$. Assume that $(g,\pi)$ and $(\bar{g}_k, \bar{\pi}_k)$ are asymptotically flat. If for any $\epsilon>0$, there exits $k_0$ so that 
\begin{align*}
	& \| g - \bar{g}_k \|_{W^{2,p}_{-q}(M \setminus R_0)} \le \epsilon, \quad \| \pi - \bar{\pi}_k \|_{W^{1,p}_{-1-q}(M \setminus R_0)} \le \epsilon, \quad \textup{for all } k \ge k_0,
\end{align*}
then the mass and linear momentum of $ ( \bar{ g }_k, \bar{ \pi }_k)$ are within $ \epsilon $ of those of $ ( g, \pi )$.
 
In addition to the above conditions, if $(g,\pi)$ and $\{(\bar{g}_k, \bar{\pi}_k)\}$ satisfy the RT condition and if there is a constant $c$ so that for $k$ large,
\begin{align} \label{eq:RT}
	\| (g-\bar{g}_k)^{\textup{odd}}\|_{W^{2,p}_{-1-q}(M\setminus B_{R_0}) } < c,\quad \| (\pi - \bar{\pi}_k)^{\textup{even}} \|_{W^{1,p}_{-2-q}(M\setminus B_{R_0})} < c,
\end{align} 
then the mass, linear momentum, center of mass, angular momentum of $ ( \bar{ g }_k, \bar{ \pi }_k)$ are within $ \epsilon $ of those of $ ( g, \pi )$.
\end{prop}
\begin{proof}
By the definition of  mass and the divergence theorem, for any $s$ large, 
\begin{align*}
	16\pi m(\bar{g}_k) &= \lim_{r \rightarrow \infty} \int_{|x|=r} \sum_{i,j} [(\bar{g}_k)_{ij,i} -(\bar{g}_k)_{ii,j} ] \frac{x^j}{|x|} \, d\sigma_0 \\
	&=  \int_{|x|=s} \sum_{i,j} [(\bar{g}_k)_{ij,i} -(\bar{g}_k)_{ii,j} ] \frac{x^j}{|x|} \, d\sigma_0 \\
	&\quad + \int_{|x| \ge s} \sum_{i,j} [(\bar{g}_k)_{ij,ij} - (\bar{g}_k)_{ii,jj}]\, dx.
\end{align*}
In the volume integral $\sum_{i,j}\left[(\bar{g}_k)_{ij,ij} - (\bar{g}_k)_{ii,jj}\right]$ is the leading order term of the scalar curvature $R(\bar{g}_k)$. Hence, by the constraint equations, the integrand is of the order $O(|x|^{-2-2q})$, so there is a constant $s_0$ such that
\[
	 \int_{|x| \ge s} \sum_{i,j} [(\bar{g}_k)_{ij,ij} -( \bar{g}_k)_{ii,jj}] \, dx \le \frac{\epsilon}{4}, \quad \textup{for all } s \ge s_0.	
\] 
To handle the boundary term, we symbolically denote 
\[
	\bigg|\int_{|x|= s}  \sum_{i,j} \left[  g_{ij,i} - g_{ii,j} - ((\bar{g}_k)_{ij,i} -(\bar{g}_k)_{ii,j})\right] \frac{x^j}{|x|}  \, d\sigma_0 \bigg|\le  \int_{|x|=s}|D(g - \bar{g}_k )| \, d\sigma_0.
\]
First, if $p\ge 3$, by the Sobolev embedding inequality \cite[(1.9)]{Bartnik87}, 
\begin{align*}
	\sup_{M \setminus B_{R_0}} | D (g - \bar{g}_k)| |x|^{1+q} &\le c \| D (g - \bar{g}_k) \|_{W^{1,p}_{-1-q}(M\setminus B_{R_0})} \\
	&\le c \| g - \bar{g}_k \|_{W^{2,p}_{-q}(M\setminus B_{R_0})}.
\end{align*}
Fix $\epsilon$ and $s_0$. There exists $k_0'$ so that 
\[
	\| g - \bar{g}_k \|_{W^{2,p}_{-q}(M\setminus B_{R_0})} \le \frac{\epsilon}{8 c \pi s_0^{1-q}}, \quad \textup{ for all } k \ge k_0'.
\]
Therefore, 
\[
	\int_{|x|=s_0}|D(g - \bar{g}_k )| \, d\sigma_0 \le  \max_{|x|=s_0} (| D (g - \bar{g}_k)| s_0^{1+q} ) 4 \pi s_0^{1-q} \le \frac{\epsilon}{2}.
\]
Combining the above estimates, we prove that $| m(g) - m(\bar{g}_k) | \le \epsilon$ for all $k \ge k_0'$. When $1 \le p \le 3$, by the definition of the weighted Sobolev norm, 
\[
	 \int_{R_0}^{\infty} \int_{|x|=r} (| D (g - \bar{g}_k)  | r^{1+q})^p r^{-3} \, d\sigma_0 r^2 dr   \le \| g - \bar{g}_k \|_{W^{2,p}_{-q} (M \setminus B_{R_0}) }^p < \infty.
\]
Therefore, for each $k$, there exists $s_k$ with $s_k\rightarrow \infty$ so that 
\[
	\int_{|x| = s_k} ( |D(g - \bar{g}_k )| s_k^{1+q} )^p s_k^{-1} \, d\sigma_0 \le \frac{\epsilon^p}{2^p} s_k^{-1/2},
\]
for otherwise the volume integral would diverge. Then by the H\"{o}lder inequality, for $\frac{1}{p^*}= 1- \frac{1}{p} $,
\begin{align*}
	\int_{|x|=s_k}|D(g - \bar{g}_k )| \, d\sigma_0  &\le \left(\int_{|x|=s_k} ( |D(g - \bar{g}_k )| s_k^{1+q} )^p s_k^{-1} \, d\sigma_0  \right)^{\frac{1}{p}} s_k^{\frac{1}{p^*} - q} \\
	&\le \frac{\epsilon}{2} s_k^{-\frac{1}{2p} + \frac{1}{p^*} - q}.
\end{align*}
Notice that $-\frac{1}{2p} + \frac{1}{p^*} - q < 0$ for $1 \le p \le 3$ and $q>1/2$. Then there exists $k_0'$ so that the above surface integral is less than $\epsilon/2$ for all $k\ge k_0'$. Combining the above estimates, we conclude  $| m(g) - m(\bar{g}_k) | \le \epsilon$. The proof of the convergence of linear momentum is similar. 

If we have additionally condition \eqref{eq:RT}, we show that center of mass converges. By the definition of center of mass and the divergence theorem,
\begin{align*}
	&16\pi m(\bar{g}) C^l(\bar{g}) \\
	&=  \int_{|x|= s} \left[ x^l\sum_{i,j} (\bar{g}_{ij,i} -\bar{g}_{ii,j} ) \frac{x^j}{|x|}- \sum_i \left(\bar{g}_{il} \frac{x^i}{|x|} - \bar{g}_{ii} \frac{x^l}{|x|}\right)\right] \, d\sigma_0\\
	& \quad + \int_{|x| \ge s} x^l \sum_{i,j}(\bar{g}_{ij,ij} - \bar{g}_{ii,jj}) \, dx\\
	&=   \int_{|x|=s} \left[ x^l\sum_{i,j} ((\bar{g}^{\textup{odd}}_{ij})_{,i} -(\bar{g}^{\textup{odd}}_{ii})_{,j} ) \frac{x^j}{|x|}- \sum_i \left(\bar{g}^{\textup{odd}}_{il} \frac{x^i}{|x|} - \bar{g}^{\textup{odd}}_{ii} \frac{x^l}{|x|}\right)\right] \, d\sigma_0\\
	& \quad + \int_{|x| \ge s} x^l \sum_{i,j}((\bar{g}^{\textup{odd}}_{ij})_{,ij} - (\bar{g}^{\textup{odd}}_{ii})_{,jj}) \, dx.
\end{align*}
The volume integral is of the order $O(s^{1-2q})$ by the constraint equations and condition \eqref{eq:RT}, so it is less than $\epsilon/4$ for all $s$ large enough. For the boundary term, we need to estimate the surface integral
\begin{align*}
	 \int_{|x|=s} \left[ |x| |D(\bar{g}_k - g)^{\textup{odd}}| + |(\bar{g}_k -g )^{\textup{odd}}| \right] d\sigma_0.
\end{align*}
We then proceed as above and show that, given $\epsilon>0$, there is a sequence $\{ s_k\}$ with $s_k \rightarrow \infty$  so that for $k$ large
\begin{align*}
	 \int_{|x|=s_k} \left[ |x| |D(\bar{g}_k - g)^{\textup{odd}}| + |(\bar{g}_k -g )^{\textup{odd}}| \right] d\sigma_0 \le \frac{\epsilon}{2}.
\end{align*}
We can then conclude that $|C^l(g) - C^l (\bar{g}_k) |\le \epsilon$. The proof for the angular momentum is analogous.
 \end{proof}
The second part of the statement in the above proposition is optimal in the sense that condition \eqref{eq:RT} is necessary. In the recent joint work with Rick Schoen and Mu--Tao Wang \cite{HSW10}, we are able to construct perturbed data $(\bar{g}_k, \bar{\pi}_k)$ close to some given asymptotically flat data with the RT condition $(g,\pi)$ in $W^{2,p}_{-q}(M) \times W^{1,p}_{-1-q}(M)$. The energy-momentum vector of the new data sets is equal to that of $(g,\pi)$, but their center of mass and angular momentum can be arbitrarily prescribed. In particular, the perturbed data sets violate condition \eqref{eq:RT}.

\section{The constant mean curvature foliation near infinity}
In the asymptotically flat region, the mean curvature $H_S$ of the coordinate sphere $S_R(p)$ is close to the constant $2/R$, which is the mean curvature of $S_R(p)$ in Euclidean space.  The difference $H_S - 2/R$ is measured by $h = g-\delta$. Because $H_S(x) = \textup{div}_g (\nu_g)$ where $\textup{div}_g$ is the divergence operator of $(M,g)$ and $\nu_g$ is the unit outward normal vector field on $S_R(p)$ with respect to $g$. By direct computations, we have
\begin{align} \label{MC}
\begin{split}
		H_S(x)&=\frac{2}{R} + \frac{1}{2} \sum_{i,j,k} h_{ij,k}(x ) \frac{ (x^i - p^i) ( x^j - p^j )( x^k  - p^k) }{ R^3} \\
				&\quad + 2 \sum_{i,j} h_{ij}(x) \frac{ ( x^i - p^i )( x^j - p^j ) }{ R^3 } - \sum_{i,j}h_{ij,i}(x ) \frac{ x^j - p^j}{R}  \\
				& \quad+ \frac{1}{2} \sum_{i,j}h_{ii,j}(x)  \frac{x^j - p^j}{ R }- \sum_i\frac{h_{ii} (x) }{R} + E_0(x), 
\end{split}
\end{align}
where $E_0(x) = O(R^{-1-2q})$ and $E_0^{\textup{odd}}(x)= O(R^{-2-2q})$. 

We first examine \eqref{MC} when $g$ is strongly asymptotically flat, i.e. $g = (1+\frac{2m}{|x|})\delta + O(|x|^{-2})$. Substituting $h_{ij} = \frac{2m}{|x|}\delta_{ij} + O(|x|^{-2}) $ into \eqref{MC} and simplifying the summation terms, we derive
\[
	H_S(x) = \frac{2}{R} - \frac{4m}{R^2} + \frac{6m (x-p)\cdot p}{R^4} + \frac{9m^2}{R^3} + O(R^{-3}+ |p|R^{-4}).
\]
The linearized mean curvature operator on $S_R(p)$ has an approximate kernel spanned by $\{ x^1- p^1, x^2 - p^2, x^3 - p^3\}$. It was then observed by Ye \cite{Ye96} that the projection of the last three terms of the right hand side above into the approximate kernel can be annihilated by choosing suitable $p$, which we later identified with center of mass \cite{Huang09a}. Then the inverse function theorem can be applied to find the constant mean curvature surface $\Sigma_R$ near the coordinate sphere $S_R(p)$. The two surfaces are close in the sense that $\Sigma_R$ is the graph of $\phi$ over $S_R(p)$ and $\|\phi\|_{C^{2,\alpha}(S_R(p))}$ is bounded by a constant uniformly in $R$. 

In our situation, because $g$ has weaker asymptotics and decay rates, the inverse function theorem cannot be directly applied because the asymptotics   of $H_S-2/R$ in \eqref{MC} may be far from being a constant. It simply reflects the fact that the coordinate spheres are defined with respect to some asymptotically flat coordinate chart which is not canonical, so the constant mean curvature surfaces may not be close to the coordinate spheres. One may consider to choose a better asymptotically flat coordinate chart, but it seems hard to handle and simplify all the summation terms in \eqref{MC}. Instead, we explicitly construct a family of approximate spheres $\mathcal{S}(p,R)$ which reflect the asymptotics better than the coordinate spheres for each $p$ and large $R$. Moreover, from our construction, the approximate spheres also adapt better the RT condition \cite{Huang10a}. 
\begin{lemma}\label{lemma:approx}
Let $(M,g)$ be asymptotically flat satisfying the RT condition.  There exists a constant $c$ independent of $R$ so that, for each $p$ and for $R$ large, there is an approximate sphere 
\begin{align*}
	\mathcal{S}(p,R) = \left\{ x + \phi (x) \nu_g : \phi \in C^{2,\alpha} (S_R(p))\right\}.
\end{align*}
Here $\phi^*$ $($the pull back of $\phi$ on $S_1(0)$$)$ satisfies
\begin{align} \label{PHI} 
	\| \phi^* \|_{C^{2,\alpha} (S_1(0))} \le c R^{1-q}, \quad \| (\phi^*)^{\textup{odd}}\|_{C^{2,\alpha} (S_1(0))} \le c R^{-q}.
\end{align}
Moreover, the mean curvature of $\mathcal{S}(p,R)$ is
\begin{align} \label{eq:approxmean}
	H_{\mathcal{S}} = \frac{2}{R} + \bar{f} + O(R^{-1-2q}),
\end{align}
where $\bar{f} := (4 \pi R^2)^{-1}\int_{S_R(p) } f\, d\sigma_0$ and $ f= H_S - 2/R$.
\end{lemma}
\begin{remark}
When $q=1$, $\phi$ is bounded by a constant, so the approximate sphere stays within a constant neighborhood of the coordinate sphere. However, when $q<1$, the size of $\phi$ may increase as $R$ increases. Nevertheless, 
$\mathcal{S}(p,R)$ has asymptotic antipodal symmetry with respect to the center $p$ as seen in \eqref{PHI}. 
\end{remark}
After showing the existence of the family of the approximate spheres, we prove that we can perturb them to construct constant mean curvature surfaces, if the center $p$ is chosen correctly. The following identity \eqref{CENTER} plays a key role to locate the center $p$ of the approximate spheres.  
\begin{lemma}\label{lemma:key}
Let $(M,g)$ be an asymptotically flat manifold satisfying the RT condition. For $l = 1,2,3,$
\begin{align}\label{CENTER}
	\int_{S_R(p)} (x^l - p^l)  \left( H_S - \frac{2}{R} \right) \,d\sigma_0 = 8\pi m (p^l -  C^l)+ O(R^{1-2q}). 
\end{align}
\end{lemma}
First, it is simple to see that if the metric $g = u^4 \delta$, then \eqref{CENTER} holds as  follows. The left hand side above is 
\begin{align*}
	&\int_{S_R(p)} (x^l - p^l)  \left( H_S - \frac{2}{R} \right) \,d\sigma_0 \\
	&= \int_{S_R(p)} (x^l - p^l) \left( \sum_{j} 4 u^3 u_{,j} \frac{x^j - p^j} {R} - \frac{u^4 -1 }{R} \right)\, d\sigma_0 + O(R^{1-2q})\\
	&= - p^l \int_{S_R(p)} \sum_j 4 u^3 u_{,j} \frac{x^j - p^j}{R} \, d\sigma_0 \\
	&\quad + \int_{S_R(p)} \left[\sum_j 4 x^lu^3 u_{,j} \frac{x^j - p^j} {R}-  u^4  \frac{x^l - p^l}{R}\right] \, d\sigma_0 + O(R^{1-2q}).
\end{align*}
On the other hand, computing $m$ and $C^l$ when $g = u^4 \delta$, we have
\begin{align*}
	 & \int_{S_R(p)} 4 u^3 u_{,j}\sum_j \frac{x^j - p^j}{R} \, d\sigma_0 = - 8\pi  m+ O(R^{1-2q}), \\
	 &\int_{S_R(p)} \left[\sum_j 4 x^lu^3 u_{,j} \frac{x^j - p^j} {R}-  u^4  \frac{x^l - p^l}{R}\right] \, d\sigma_0 = -8\pi m C^l + O(R^{1-2q}).
\end{align*}
We then conclude that \eqref{CENTER} holds for $g = u^4 \delta$.  For general metric $g$ satisfying the RT condition, one may tend to apply Theorem \ref{DenThm2} at this stage. However, it seems that we can only prove that \eqref{CENTER} holds for a sequence of radii $\{R_i\}$ up to some error $\epsilon$. In other words, it only allows us to conclude that 
\[
	\lim_{ R \rightarrow \infty }\int_{S_R(p)} (x^l - p^l)  \left( H_S - \frac{2}{R} \right) \,d\sigma_0 = 8\pi m (p^l -  C^l). 
\]
However, for our purpose to construct the constant mean curvature surface at each radius $R$, we need to show that \eqref{CENTER} holds for each $R$. Therefore, we study the summation terms in \eqref{MC} separately and observe some cancellations in the integral in \eqref{CENTER} after applying the divergence theorem.  In particular, our argument shows that \eqref{CENTER} holds when the asymptotically flat coordinates $\{x\}$ are global, without using the density argument. When the coordinates $\{x\}$ are defined only outside the compact set, the (outer) boundary term is then handled by the density theorem (Theorem \ref{DenThm2}).  The proof is computationally involved, so we omit  the detailed computations and present only the sketch of the proof below.  
\begin{spf}
We define, for $l= 1,2,3$,
\begin{align*}
				\mathcal{I}^l_g (R)  =  \int_{S_R(p)} ( x^{ l } - p^{ l} ) \left[  \frac{1}{2} \sum_{i,j,k} h_{ij,k}(x) \frac{ ( x^i - p^i )(x^j - p^j) (x^k - p^k)  }{ R^3 } \right] \, d \sigma_0.
\end{align*}
Because the asymptotically flat coordinates may not be defined in the interior, we use the divergence theorem in the annulus $A = \{ R \le |x-p| \le R_1\}$. Using integration by parts and simplifying the expression, we obtain an identity containing purely the boundary terms
\begin{align} \label{eq:boundary}
		\mathcal{I}^l_g (R_1) - \mathcal{I}^l_g (R) = \mathcal{B}^l_g (R_1) - \mathcal{B}^l_g (R) \quad \textup{ for all } R_1 \ge R,
\end{align}
where $\mathcal{B}^l_g (R)$ equals the boundary integrals: 
\begin{align*}
		&\int_{S_R(p)} ( x^l - p^l) \sum_{i,j} \left[\frac{1}{2} h_{ij,i}(x) \frac{ x^j - p^j }{ R } - 2 h_{ij}(x)\frac{ ( x^i - p^i)( x^j - p^j )  }{ R^3}\right] \,d \sigma_0 \\
		&+ \int_{S_R(p)}\frac{1}{2}  \sum_i \left[h_{ii} (x)\frac{ x^l - p^l  }{ R} +  h_{i l }(x) \frac{ x^i - p^i }{ R}  \right] d \sigma_0. 
\end{align*}
We would like to show that $\mathcal{I}^l_g(R)= \mathcal{B}^l_g (R)$ for each $R$ large and for $l=1,2,3$. It suffices to prove that 
\[
	\lim_{R_1 \rightarrow \infty} \mathcal{I}^l_g (R_1) = \lim_{R_1\rightarrow \infty}\mathcal{B}^l_g (R_1).
\]
First notice that if $g = u^4 \delta$ outside a compact set, then by direct computations, for any $R_1$ large (so that $g=u^4 \delta$ on $B_{R_1}(p)$),  
\[
	\mathcal{I}^l_g (R_1) = \mathcal{B}^l_g (R_1).
\]
To prove the identity for general metrics $g$, we can apply Theorem \ref{DenThm2}. There exists a sequence of data $\{ (\bar{g}_k, \bar{\pi}_k)\}$ with harmonic asymptotics and  radii $\{ R_k\}$ with $R_k \rightarrow \infty$. Given $\epsilon>0$, there exists $k_0'$ so that
\begin{align*}
	 |\mathcal{I}^l_g (R_k) -  \mathcal{I}^l_{\bar{g}_k}(R_k) |\le \epsilon,\quad |\mathcal{B}^l_g (R_k)-  \mathcal{B}^l_{\bar{g}_k} (R_k) | \le \epsilon \quad \mbox{for all } k \ge k'_0.
\end{align*}
The proof of the estimates is similar to the argument given in the proof of Proposition \ref{prop:cont}. Therefore, 
\[
	\lim_{k \rightarrow \infty} \mathcal{I}^l_g (R_k) = \lim_{k\rightarrow \infty} \mathcal{B}^l_g (R_k).
\] 
We then conclude that $\mathcal{I}^l_g(R)= \mathcal{B}^l_g (R)$ for $l=1, 2, 3$. Substituting $\mathcal{I}_g^l (R)$ by $\mathcal{B}^l_g(R)$ into \eqref{MC} and \eqref{CENTER},  and simplifying the expression, we have
\begin{align*}
		&\int_{S_R(p)} ( x^l - p^l ) \left( H_S - \frac{2}{R}\right) \, d\sigma_0 \\
		 &= - \frac{1}{2} \left[  \int_{S_R(p)}  (x^l - p^l ) \sum_{i,j} ( h_{ij,i} - h_{ii,j} )\frac{ x^j - p^j }{ R } \, d\sigma_0 \right.\\
		&\quad \left.-  \int_{S_R(p)} \sum_i \left(  h_{il} \frac{ x^i - p^i }{ R}- h_{ii} \frac{ x^l - p^l }{ R} \right) \,d\sigma_0 \right] + O(R^{1-2q}).
\end{align*}
Using the definitions of mass (\ref{eqn:mass}) and center of mass \eqref{def:ctm}, we derive \eqref{CENTER}.
\qed
\end{spf}

To construct constant mean curvature surfaces from the approximate spheres $\mathcal{S}(p,R)$, we consider the mean curvature operator
\[
	\mathcal{H}: C^{2,\alpha} (\mathcal{S}(p,R)) \rightarrow C^{0,\alpha}(\mathcal{S}(p,R))
\]
defined by $\mathcal{H}(\psi)$ equal to the mean curvature of the graph of $\psi$ over $\mathcal{S}(p,R)$. We would like to construct $\psi \in C^{2,\alpha} (\mathcal{S}(p,R))$ so that $\mathcal{H}(\psi)$ equals a constant which depends only on $p$ and $R$. By observing \eqref{eq:approxmean}, it is natural to set the constant equal to $2/R + \bar{f}$. By Taylor's formula, 
\begin{align} \label{MAIN}
\begin{split}
	\mathcal{H}(\psi ) =& \mathcal{H}( 0) + \Delta_{\mathcal{S}} \psi + \left( | A |^2 + Ric^M (\mu_g, \mu_g)\right) \psi \\ 
	&+ \int_0^1 \left( d \mathcal{H} ( s \psi ) - d \mathcal{H} ( 0)\right) \psi \, ds,
\end{split}
\end{align}
where $\Delta_{\mathcal{S}}$ is the Laplace--Beltrami operator and $A$ is the second fundamental form of $\mathcal{S}(p,R)$ in $M$, $ Ric^M$ is the Ricci curvature tensor of $(M,g)$, and $\mu_g$ is the outward unit normal vector of $\mathcal{S}(p,R)$. Because $\mathcal{S}(p,R)$ is close to $S_p(R)$ in the asymptotically flat region, 
\[
	| A |^2 = \frac{2}{R^2} + O(R^{-2-q}), \quad |Ric^M| = O(R^{-2-q}).
\]
Therefore, by \eqref{eq:approxmean}, to solve $\mathcal{H}(\psi)=2/R + \bar{f}$ is equivalent to solving
\begin{align} \label{eq:linear}
	\Delta_0 (\psi\circ \Phi) + \frac{2}{R^2} (\psi\circ \Phi )= \bar{f} - f+ \textup{ lower order terms},
\end{align}
where $\Delta_0$ is the standard spherical Laplacian on $S_R(p)$, and the diffeomorphism $\Phi : S_R(p) \rightarrow \mathcal{S}(p,R)$ is defined by $\Phi(x) = x + \phi \nu_g$. This operator $\Delta_0 + 2/R^2$ has the kernel spanned by $\{ x^1-p^1, x^2 - p^2, x^3 - p^3\}$. A necessary condition to solve \eqref{eq:linear} is that the projection of the right hand side in \eqref{eq:linear} into the kernel has to be zero. By Lemma \ref{lemma:key}, if $m$ is nonzero, we can choose the center $p$ so that the projection is zero. Moreover, the center $p = C+ O(R^{1-2q})$ is asymptotic to center of mass.

Finally, \eqref{eq:linear} is then a quasi-linear elliptic equation. We can apply the Schauder fixed point theorem to find a solution $\psi$ and prove the following theorem.

\begin{theorem}[\cite{Huang10a}]\label{THM1}
Assume that $(M, g)$ is asymptotically flat satisfying the RT condition at the decay rate $q$ greater than $1/2$. If $m$ is nonzero, then there exist surfaces $\{ \Sigma_R \}$ with constant mean curvature $H_{\Sigma_R}= (2 / R) + O(R^{-1-q})$ for $R$ large. Moreover, $\Sigma_R$ is a $c_0 R^{1-q}$-graph over $S_R (C)$, i.e. 
\[
	\Sigma_R =\left \{ x +  \psi_0 (x) \nu_g :  \psi_0 \in C^{2,\alpha}(S_R(C)) \right\}
\] 
with 
\[
	\| \psi_0^*\|_{C^{2,\alpha} (S_1 (0))} \le c_0 R^{1-q}, \quad \textup{and} \quad \| (\psi_0^*)^{\textup{odd}}\|_{C^{2,\alpha} (S_1 (0))} \le c_0 R^{-q}, 
\]
where $\psi_0^*$ is the pull back of $\psi_0$ on $S_1(0)$. 

Moreover, the geometric center defined by
\begin{align} \label{eq:geoctm}
	\lim_{R\rightarrow \infty} \frac{\int_{\Sigma_R} x \, d\sigma_0}{ \int_{\Sigma_R} \, d\sigma_0}
\end{align}
is equal to the center of mass \eqref{def:ctm}.
\end{theorem}
Furthermore, the family of constant mean curvature surfaces constructed above form a foliation, if mass is positive.  
\begin{theorem}[\cite{Huang10a}]\label{thm:foli}
Under the same assumption as in Theorem \ref{THM1}, if in addition $m > 0$, then each $\Sigma_R$ is strictly stable and $\{ \Sigma_R \}$ form a foliation. 
\end{theorem}
We first recall that a closed constant mean curvature surface $\Sigma$ is \emph{stable} if 
\begin{align} \label{eq:stable}
	\int_{\Sigma} u L_{\Sigma} u \, d\sigma_g = \int_{\Sigma}\left[ | \nabla u|^2 - (|A|^2 + Ric^M (\nu_g, \nu_g)) u^2 \right] \, d\sigma_g \ge 0
\end{align}
for all function $u$ with $\int_{\Sigma} u \, d\sigma_g = 0$, where $L_{\Sigma}  = - \Delta_{\Sigma} - (|A|^2 + Ric^M(\nu_g, \nu_g))$ is the stability operator. Geometrically, a stable constant mean curvature surface minimizes the area among nearby surfaces which enclose the same volume.  If the inequality in \eqref{eq:stable} is strict, $\Sigma$ is called \emph{strictly} stable.
\begin{spf}
To prove that each $\Sigma_R$ constructed in Theorem \ref{THM1} is strictly stable, one would like to estimate the eigenvalue of the stability operator $L_{\Sigma_R}$. Because $\Sigma_R$ is roughly spherical, $L_{\Sigma_R}$ has an approximate kernel spanned by functions whose mean values are almost zero. To show that $L_{\Sigma_R}$ indeed has no kernel, one has to estimate the term of $(|A|^2 - 2/R^2 )+ Ric^M(\nu_g, \nu_g)$ whose pointwise value is however undetermined by the condition on the asymptotics of $g$. We observe that by the Bochner--Lichnerowicz identity, we only need to derive the integral estimates. Hence, after some cancellations, the lowest eigenvalue $\lambda_1$ is then estimated by
\[
	\lambda_1 \ge -\frac{3}{4\pi R^2} \int_{\Sigma_R} Ric^M(\nu_g, \nu_g) \, d\sigma_g + O(R^{-2-2q}).
\]
Furthermore, we can relate the integral to mass and obtain
\begin{align*}
	\int_{\Sigma_R}  Ric^M(\nu_g, \nu_g)  \, d\sigma_g &= \int_{S_R(p)} Ric^M (\nu_0, \nu_0) \, d\sigma_0 + O(R^{-1-q}) \\
	&= -\frac{8\pi m}{R} + O(R^{-1-q}),
\end{align*}
where the second inequality follows from an alternative definition of mass (see \eqref{ADMmass} in the next section). Therefore, when $m$ is positive, the lowest eigenvalue $\lambda_1$ of $L_{\Sigma_R}$ among functions with zero mean value is positive for $R$ large, and then $\Sigma_R$ is strictly stable.

Furthermore, $L_{\Sigma_R}$ is indeed invertible. It is easy to see that the lowest eigenvalue $\lambda_0$ of $L_{\Sigma}$ among \emph{all} test functions is negative by letting the test functions equal to constants. We then employ the beautiful analysis of Huisken and Yau \cite{HY96} in our setting  to show that the next eigenvalue is positive. Hence $L_{\Sigma_R}$ is invertible. Then we apply the inverse function theorem for the mean curvature operator and conclude that $\{\Sigma_R\}$ vary smoothly in $R$ and that locally $\Sigma_{R'}$ is graph of $v$ over $\Sigma_R$ when $R'$ is close to $R$.  To prove that they form a foliation, it remains to show that $\Sigma_R$ and $\Sigma_{R'}$ do not intersect if $R\neq R'$; that is, $v$ is either strictly positive or negative. To see this, we decompose $v$ into the projection on the lowest eigenfunction and the orthogonal complement. We derive that the lowest eigenfunction is close to a nonzero constant by the De Giorgi--Nash--Moser theory. Then by the Schauder estimates, we prove that the norm of the orthogonal complement is relatively small. Hence, $v$ is a nonzero constant up to the lower order terms $O(R^{-q})$.  
\qed
\end{spf}

In order to make the geometric center of mass canonical and to define a polar coordinate system outside a compact set in $M$, one would prove that the constant mean curvature foliation constructed is globally unique. The inverse function theorem in the proof above already gives uniqueness of the foliation locally near $\Sigma_R$ for each $R$. For a more global result, we have the following theorem. Below, we denote by $\underline{r} = \min \{ |x| : x \in N\}$ the minimum radius of $N$.
\begin{theorem} [\cite{Huang10a}] \label{GU1}
Assume that $(M, g)$ is asymptotically flat with the RT condition at the decay rate $q$ greater than $1/2$ and $m>0$. Then there exists $\sigma_1$ so that if $N$ has the following properties:
\begin{enumerate}
\item $N$ is topologically a sphere,
\item $N$ has constant mean curvature $H = H_{\Sigma_R}$ for some $R \ge \sigma_1$,
\item $N$ is stable,
\item $\underline{r} \ge  H^{-a}$ for some $a$ satisfying $\displaystyle \frac{5-q}{2(2+q)} < a \le 1$,
\end{enumerate}
then $N = \Sigma_R$.
\end{theorem}
 When $q=1$, the range of $a$ coincides with the a priori estimates by Metzger \cite{Metzger07}. When $q<1$,  $a$ has a more narrow range. The optimal result that one could hope for is $a=0$; namely, the foliation is unique outside a fixed compact set. We remark that it is proven true for strongly asymptotically flat manifolds by Qing and Tian \cite{QT07}, and it remains open for general asymptotics. \footnote{After this paper was submitted to publish, we noticed an announcement on the uniqueness result by Ma \cite{M}.}

\section{Equivalence of center of mass}
The constant mean curvature foliation constructed in the previous section gives a geometric center of mass. By our construction, it is easy to see that the geometric center of mass \eqref{eq:geoctm} is equal to the classical Hamiltonian notion \eqref{def:ctm} because each constant mean curvature surface is roughly round centered at $p = C + O(R^{1-2q})$. In addition to the Hamiltonian and geometric notions of center of mass, another notion using conformal Killing vector fields proposed by R. Schoen is as follows:
\begin{align} \label{CTM}
		C_I^l = \frac{1}{ 16\pi m}\lim_{ r \rightarrow \infty} \int_{ |x| = r } \sum_{i,j} \left( R^M_{ij} - \frac{1}{2} R_g g_{ij} \right) Y_{(l)}^i \nu_g^j d{\sigma}_g , \quad l = 1, 2, 3 
\end{align}
where $Y_{(l)} =  \sum_i \big(  | x |^2 \delta^{l i} - 2 x^{l} x^i \big) \frac{\partial}{\partial x^i} $ are Euclidean conformal Killing vector fields. If the metric $g$ satisfies the RT condition, we can replace $\nu_g^j$ and $d \sigma_g$ in \eqref{CTM} by $\frac{x^j}{r}$ and $d\sigma_0$ respectively. This intrinsic definition is motivated by a similar expression of the mass when $ Y_{(l)}$ is replaced by $\{ - 2 x^{i}\frac{\partial  }{ \partial x^i}\}$, the radial direction Euclidean conformal Killing vector field:
\begin{align} \label{ADMmass}
		m_I = \frac{1}{ 16\pi } \lim_{r\rightarrow \infty} \int_{ |x| = r}  \sum_{i,j} \left( R^M_{ij} - \frac{1}{2} R_g g_{ij} \right) ( -2 x^i) \nu_g ^j d{\sigma}_g.
\end{align}
Notice that when $g$ is asymptotically flat, we can replace $\nu_g^j$ and $d \sigma_g$ in \eqref{ADMmass} by $\frac{x^j}{r}$ and $d\sigma_0$ respectively. The Euclidean conformal Killing vector fields $\{ - 2 x^{i}\frac{\partial  }{ \partial x^i}\}$ and $ Y_{(l)}$ generate dilation and translation near infinity. A detailed discussion on these conformal Killing vector fields can be found in, for example \cite{CW08}. 

One way to prove \eqref{CTM} and \eqref{ADMmass} is to apply the density theorems in Section 2. It seems that the proof of $m=m_I$ has not been published anywhere, so we present the proof below.
\begin{prop}
The notion of mass \eqref{ADMmass} is equal to the classical definition \eqref{eqn:mass}, i.e.
\[
	m_I=m.
\]
Similarly, \eqref{CTM} is equal to the classical notion of center of mass \eqref{def:ctm}, i.e.,
\[
	C^l_I = C^l, \quad \textup{for } l = 1, 2, 3.
\]
\end{prop}
\begin{proof}
To prove $m=m_I$, we first consider the metric $g$ where  $g= u^4 \delta$ outside a compact set. Because $u$ is asymptotically harmonic, its expansion is $1+\frac{A}{|x|} + O(|x|^{-2})$ for some constant $A$. From \eqref{eqn:mass} and direct computations, $A = m/2$. Then, the expansion of the conformal factor $u$ is
\begin{align*}
	& u = 1 + \frac{m}{2 |x| } + O(|x|^{-2})\\
	& \partial_i u  = \frac{- m x^i}{2|x|^3} + O(|x|^{-3}), \quad \partial_i \partial_j u = -\frac{m\delta_{ij} }{2|x|^3} + \frac{3m x^i x^j}{2|x|^5} + O(|x|^{-4}).
\end{align*}
On the other hand, because $g$ is conformal to the Euclidean metric, the Ricci curvature of $g$ is
\[
	R_{ij} = - 2 u^{-1} \partial_i \partial_j u + 6 u^{-2} \partial_i u \partial_j u - 2(u^{-1} \Delta u + u^{-2} | \nabla u|^2)\delta_{ij}.
\]
Substituting the expansion of $u$ into the above formula, we get
\[
	R_{ij} (-2x^i) x^j = \frac{4m}{|x|^2} + O(|x|^{-3}).
\]
Also notice that $R = -4 u^{-1} \Delta u = O(|x|^{-4})$. Therefore, by \eqref{ADMmass},
\begin{align*}
		m_I = \frac{1}{ 16\pi } \lim_{r\rightarrow \infty} \int_{ |x| = r}\left[ \frac{4m}{|x|^2} + O(|x|^{-3})\right] \, d\sigma_0 = m.
\end{align*}
For the asymptotically flat metric $g$ with general asymptotics, we can apply the the density argument by Schoen and Yau \cite{SY81a} (or alternatively Theorem \ref{DenThm}): Let $\{\bar{g}_k\}$ be a sequence of metrics with harmonic asymptotics. We have 
 \[
 	\lim_{k\rightarrow \infty} m(\bar{g}_k) =m(g).
\]
The same argument in the proof of Proposition \ref{prop:cont} would also imply 
 \[
 	\lim_{k\rightarrow \infty} m_I(\bar{g}_k) =m_I(g).
\]
Because $m(\bar{g}_k) = m_I (\bar{g}_k)$ as computed above, we complete the proof.

The proof of $C^l = C_I^l$ can be proceeded similarly (see \cite[Theorem 1]{Huang09a}). 
\end{proof}

To see that the expression \eqref{CTM} is intrinsic, we prove that the integral surfaces can be replaced by more general surfaces, which are defined geometrically \cite{Huang09a}. 
\begin{prop} \label{GeneSur}
Suppose that $( M , g) $ is asymptotically flat with the RT condition. Let $\{ D_k \}_{k = 1}^{\infty} \subset M$ be closed sets such that the sets $ S_k = \partial D_k$ are connected $C^1$ surfaces without boundary which satisfy
\begin{align*}
&\qquad    (1) \; r_k = \inf \{ |x| : x\in S_k \} \rightarrow \infty \quad \textup{as } k \rightarrow \infty,&\\
&\qquad    (2) \; r_k^{-2} {\rm area}(S_k)  \textup{ is bounded as }  k \rightarrow \infty,&\\
&\qquad    (3) \;{\rm vol}\{ D_k \setminus D_k^-\} = O( r_k ^{3^-}), \notag & \\
&\qquad \quad   \textup{where } \,D_k^- = D_k \cap \{- D_k\} \textup{ and  $3^-$ is a number less than $3$}. &
\end{align*}
Then the center of mass defined by 
\begin{align*}
		C_I^l = \frac{1}{ 16\pi m} \lim_{ k \rightarrow \infty} \int_{ S_k } \sum_{i,j}\left( R^M_{ij} - \frac{1}{2} R_g g_{ij} \right) Y_{(l)}^i \nu_g^j \,d{\sigma}_g, \quad l = 1, 2, 3
\end{align*}
is independent of the sequence $\{ S_k \}$.
\end{prop}
The first two conditions (1) and (2) on $S_k$ are the conditions also considered by Bartnik \cite{Bartnik87} to ensure the mass is well-defined. The volume growth condition (3) allows us to consider a general class of surfaces which are roughly symmetric; namely,  the non-symmetric region $D_k \setminus D_k^-$ of $D_k$ has the volume growth slightly less than that of arbitrary regions in $M$. 

Last, we can prove that the center of mass is well-defined for a fixed time-slice. That is, if we compute center of mass with respect to two different asymptotically flat charts, say $\{x\}$ and $\{y\}$, the center of mass transforms properly. Assume that $F$ is the transition function between $\{x\}$ and $\{y\}$ and $y = F(x)$. Bartnik proved that the only possible coordinate changes for asymptotically flat manifolds at infinity are rotation and translation \cite[Corollary $3.2$]{Bartnik87}. More precisely, there is a rigid motion of $\mathbb{R}^3, (\mathcal{O}^i_j, a) \in O( 3, \mathbb{R} ) \times \mathbb{R} $ so that 
$$
		\left| F(x) - ( \mathcal{O} x + a ) \right| \in W^{2, \infty}_{0} (\mathbb{R}^3 \setminus B_{R_0}).
$$
\begin{theorem} \label{ChangeCoord}
Let $\{ x \}$ and $ \{  y \}$ be two asymptotically flat coordinate charts of $(M,g)$ with the RT condition as above. Assume that $ C_{I, x} $ and $ C_{I,y} $ are the centers of mass defined by \eqref{CTM} with respect to $\{x\}$ and $\{y\}$ respectively. Then 
$$ 
		C_{I, y} = \mathcal{O} C_{I, x} + a.
$$ 
\end{theorem}
An interesting phenomenon is that compared to rotation, translation is a more subtle rigid motion. If the translation vector $a$ is zero, the above formula follows directly from the definition of $C_I$ and Proposition \ref{GeneSur}. If $a$ is nonzero, Theorem \ref{DenThm2} is involved in the proof.

\end{document}